\documentclass[12pt]{amsart}
\usepackage{verbatim}
\usepackage{eucal,url,amssymb,stmaryrd,enumerate,amscd,verbatim}
\usepackage{amsmath}
\usepackage[pagebackref,colorlinks=true,linkcolor=blue,citecolor=blue]{hyperref}
\usepackage{amsfonts}
\usepackage{amsthm}
\usepackage[margin=1in]{geometry}

\allowdisplaybreaks
\numberwithin{equation}{section}
\newtheorem{thrm}{Theorem}[section]
\newtheorem{lemma}[thrm]{Lemma}

\newtheorem{cor}[thrm]{Corollary}
\newtheorem{dfn}[thrm]{Definition}

\newtheorem{conv}[thrm]{Convention}

 1

\def\gr{\nabla f}
\def\bi{\nabla}

\begin{document}

\begin{abstract}
In this paper we establish an analogue of the classical Lichnerowicz' theorem giving a sharp lower bound of the first non-zero  eigenvalue of the sub-Laplacian on a compact seven-dimensional quaternionic contact manifold, assuming a lower bound  of the qc-Ricci tensor, torsion tensor and its distinguished  covariant derivatives.

\end{abstract}

\keywords{Quaternionic Contact Structures, Sub-Laplacian, First Eigenvalue, Lichnerowicz Inequality, 3-Sasakian}
\subjclass[2010]{53C21,58J60,53C17,35P15,53C25}
\title[A Lichnerowicz-type result]{A Lichnerowicz-type result on a seven-dimensional quaternionic contact manifold} 

\date{\today }
\author{Alexander Petkov}
\address[Alexander Petkov]{University of Sofia, Faculty of Mathematics and
Informatics, blvd. James Bourchier 5, 1164, Sofia, Bulgaria}
\email{a\_petkov\_fmi@abv.bg}
\maketitle
\tableofcontents


\setcounter{tocdepth}{2}

\section{Introduction}
The aim of this paper is to prove a seven-dimensional version of the main result established in \cite{IPV1}. Namely, we give a sharp lower bound of the first non-zero eigenvalue of the sub-Laplacian on a compact seven-dimensional quaternionic contact (abbr. QC) manifold, assuming some condition on the qc-Ricci tensor, torsion tensor and its derivatives. We pay attention to the fact, that a similar result has established in our resent paper \cite{IPV3}, in which it is concerned the so called P-function and its non-negativity for any eigenfunction. 

The problem concerning the sharp estimation of the first eigenvalue of the sub-Laplacian arises from the classical Lichnerowicz' theorem \cite{Li}, giving a sharp lower bound of the first eigenvalue of the (Riemannian) Laplacian on a compact Riemannian manifold, assuming some a-priori estimate on the Ricci tensor. More precisely, it was shown in \cite{Li} that for every compact Riemannian manifold $(M,g)$ of dimension $n$ for which the a-priori estimate 
\begin{equation}\label{Lichcon}
Ric(X,Y)\geq (n-1)g(X,Y)
\end{equation}
holds true, the first positive eigenvalue $\lambda_1$ of the Laplacian satisfies the sharp estimate 
\begin{equation}\label{shest1}
\lambda_1\geq n.
\end{equation} 
The above estimate is sharp in the sense that the equality is attained on the round unit $n$-dimensional sphere $S^n(1).$

In a natural way, a similar question arises in the sub-Riemannian geometry. Recently, a number of  Lichnerowicz-type results are established in the CR case. All of  them are provoked by the Greenleaf's work  \cite{Gr}, in which it is  obtained a Lichnerowicz-type result for a $(2n+1)$-dimensional CR manifold, $n\geq 3$.  Subsequently, the above result was extended to the case $n=2$ in \cite{LL}, where the authors have used the Greenleaf's method. Another, more restrictive result can be found in \cite{Bar}. In the quaternionic contact geometry a sharp estimate of the first eigenvalue of the sub-Laplacian is established in \cite{IPV1} for the $(4n+3)$-dimensional QC manifolds, $n\geq 2$.

The situation is more delicate in the lowest dimensions in the CR geometry and the QC geometry. The reason that this happens is that in the low-dimensional geometries appear some additional difficulties, which require a different geometric analysis, see \cite{IMV, IMV1} for the QC case. In the CR, as well as in the QC low-dimensional geometries it is necessary to be involved some different methods in comparison with these in the bigger dimensions. An exception to the rule is the conformal flatness problem, where there are no differences between the seven and the bigger dimensional  cases in the QC geometry, in contrast to the CR geometry, see \cite{Car,ChM,IVZ,IV}. In the three-dimensional CR geometry a sharp estimate is obtained in \cite{Chi06}, where, in contrast to the bigger dimensions, the author involves the CR-Paneitz operator and imposes the additional assumption for its non-negativity (some related results in the CR geometry appear in \cite{CC09a,CC09b,CC07,CCC07} and \cite{ChW}). In the seven-dimensional QC geometry a similar result has established in \cite{IPV3}, where the authors introduce a non-linear $C$ operator, motivated by the Paneitz operators,  which appear in the Riemannian and the CR geometries. Precisely, the next theorem holds.
\begin{thrm}\label{mainpan}
\cite{IPV3} Let $(M,g,\mathbb{Q})$ be a compact quaternionic contact
manifold of dimension seven. Suppose there is a positive constant $k_0$ such
that 
the qc-Ricci tensor $Ric$ and the torsion tensor $T^0$ satisfy the Lichnerowicz
type inequality
\begin{equation}  \label{7Dcondm-app1}
Ric(X,X)+6T^0(X,X)\geq k_0 g(X,X)
\end{equation}
for every horizontal vector field $X$. If, in addition, the $P-$function of
any eigenfunction of the sub-Laplacian is non-negative, then for any
eigenvalue $\lambda$ of the sub-Laplacian $\triangle$ we have the inequality
\begin{equation}\label{shestim3}
\lambda \ge \frac13k_0.
\end{equation}
\end{thrm}
 Moreover, another proof of the main result in \cite{IPV1} is given in \cite{IPV3} via the (established) non-negativity of the $P$-function in the higher dimensions.

 Another Lichnerowicz-type result in the $3$D CR geometry is proved in \cite{LL}, where the Ricci tensor, the torsion tensor and some its covariant derivatives partake in the a-priori condition. The main result of the present paper is namely a QC analog of the upper result.

Our main result follows.

\begin{thrm}\label{Main}
Let $(M,g,\mathbb{Q})$ be a seven-dimensional compact quaternionic contact
manifold. Suppose there exists a positive constant $k_0$ such
that the qc-Ricci tensor $Ric$ and the torsion tensor $T^0$ satisfy the Lichnerowicz
type inequality
\begin{equation}  \label{7Dcondm-app}
Ric(X,X)-2T^0(X,X)-\frac{36}{k_0}A(X)\geq k_0 g(X,X)
\end{equation}
for any horizontal vector field $X$, where 
\begin{equation*}
\begin{aligned}
A(X)\stackrel{\text{def}}{=}&\sum_{s=1}^3\Big[\frac{1}{6}(I_sX)^2S+2|T(\xi_s,X)|^2-\frac{2}{9}I_sX\Big((\nabla_{e_a}T^0)(e_a,I_sX)\Big)\\+&\frac{1}{6}I_sX\Big((\nabla_{e_a}T)(\xi_u,e_a,I_tX)-(\nabla_{e_a}T)(\xi_t,e_a,I_uX)\Big)-(\nabla_{\xi_s}T)(\xi_s,X,X)\Big].
\end{aligned}
\end{equation*} Then for the first nonzero eigenvalue $\lambda$ of the sub-Laplacian the next sharp estimate holds true
\begin{equation}\label{shest2}
\lambda \ge \frac13k_0.
\end{equation}
\end{thrm}

The torsion tensor $T^0$, the QC-Ricci tensor $Ric$ and the normalized QC-scalar curvature $S$ are defined in \eqref{Tcompnts} and \eqref{qscs}. See Convention~\ref{conven} about the summation rules in the definition of the function $A(X)$.

Another natural question that arises from the Riemannian geometry is stadying the case of equality in the estimate \eqref{shest2} of Theorem~\ref{Main}. The corresponding problem in the Riemannian case was considered by Obata \cite{O3}. More precisely, as a consequence of his more general result, it can be stated, that the equality in \eqref{shest1} is attained iff  the Riemannian manifold $(M,g)$ is isometrical to the unit sphere $S^n(1)$ endowed with the round metric, as \eqref{Lichcon} holds. This result has provoked a similar question in the sub-Riemannian geometry and in particular in the CR geometry, where the problem is successfully solved, see \cite{IVO,IV3,LW1}. 

The corresponding question in the QC geometry is completely resolved for the bigger dimensions ($dim M\geq 11$) in \cite{IPV2}, but it remains still open in the seven-dimensional case, except of the $3-$Sasakian case \cite[Corollary~1.2]{IPV3}, where it is shown that the minimal possible eigenvalue of the sub-Laplacian is attained only on the standard unit $3-$Sasakian sphere (up to a QC-equivalence). 

In \cite{IMV2} the authors describe explicitly the eigenfunctions corresponding to the first eigenvalue of the sub-Laplacian on the standard unit $3-$Sasakian sphere. 

In connection with the studying of the equality cases in the estimates \eqref{shestim3} and \eqref{shest2} we get as a simply consequence from Theorem~\ref{mainpan} and Theorem~\ref{Main} the next
\begin{cor}\label{corol}
Let $(M,g,\mathbb{Q})$ be a compact quaternionic contact manifold of dimension seven and $f$ be an arbitrary eigenfunction of the first eigenvalue $\lambda$ of the sub-Laplacian. Assume that some of the next a-priori conditions holds:
\begin{itemize}
\item [a)] The inequality \eqref{7Dcondm-app1} is satisfied and $T^0(\nabla f,\nabla f)\geq0$ (resp. $T^0(\nabla f,\nabla f)\leq 0$).
\item [b)] The inequality \eqref{7Dcondm-app} is satisfied and $2T^0(\nabla f,\nabla f)-\frac{36}{k_0}A(\nabla f)\geq 0$ (resp. $2T^0(\nabla f,\nabla f)-\frac{36}{k_0}A(\nabla f)\leq 0$).
\end{itemize}
If, additionally, $\lambda$ takes its minimal possible value, $\lambda=\frac{1}{3}k_0$, then the next sharp estimate
\begin{equation}\label{shestS}
S\leq\frac{k_0}{6}\quad(\textit{resp.}\quad S\geq\frac{k_0}{6})
\end{equation} 
holds.
\end{cor}
In order to simplify the exposition, we state the next
\begin{conv}\label{conven} 
Throughout this paper we shall suppose that:

\begin{enumerate}[ a)]

\item $X,Y,Z,U$ denote horizontal vector fields, i.e. $%
X,Y,Z,U\in \Gamma(H)$, while $A,B,C,D$ denote arbitrary vector fields, i.e. $A,B,C,D\in\Gamma(TM);$

\item $\{e_1,\dots,e_{4n}\}$ denotes a local orthonormal basis of the
horizontal distribution $H$;

\item if two equal vectors from the basis $%
\{e_1,\dots,e_{4n}\}$ appear in a given formula, then we have summation over them. For example, for a (0,4)-tensor $P$, the
formula $k=P(e_b,e_a,e_a,e_b)$ means $ k=\sum_{a,b=1}^{4n}P(e_b,e_a,e_a,e_b);$

\item the triples $(i,j,k)$ and $(s,t,u)$ denote  cyclic permutations of $(1,2,3)$;

\item $s$ will be a number from the set $\{1,2,3\}$, $s\in\{1,2,3\} $.
\end{enumerate}
\end{conv}

\section{Preliminaries on the quaternionic contact geometry}

The quaternionic contact structures were introduced by O. Biquard \cite{Biq1}. One can think these are quaternionic analogues of the CR structures. We refer the reader to \cite{IMV}, \cite{IV} and \cite{IV2} for comprehensive exposition and further results.

\subsection{Quaternionic contact manifolds and the Biquard connection}
\begin{dfn} A quaternionic contact (QC) structure on a $(4n+3)$-dimensional manifold $M$ is the data of co-dimension three distribution $H$ on $M$ (which is called horizontal space), locally given as the kernel of a $1$-form $\eta=(\eta_1,\eta_2,\eta_3)$ (the contact form) with values in $\mathbb{R}^3$, $H=Ker(\eta),$ which satisfy:
\begin{enumerate}
\item $H$ is equipped with an $Sp(n)Sp(1)$-structure, i.e. there exist a Riemannian metric $g$ on $H$ and  a rank three bundle $\mathbb{Q}$ consisting of endomorphisms on $H$, locally generated by the three almost complex structures $I_s:H\rightarrow H, s=1,2,3,$ satisfying the quaternionic identities: $I_1^2=I_2^2=I_3^2=-id_{|H},\quad I_1I_2=-I_2I_1=I_3$, and which are Hermitian compatible with the metric: $g(I_s\cdot,I_s\cdot)=g(\cdot,\cdot);$
\item the compatibility conditions $$2g(I_sX,Y)=d\eta_s(X,Y), s=1,2,3,$$ hold.
\end{enumerate}
A manifold $M$, endowed with a QC structure is called a quaternionic contact (QC) manifold, and is  denoted by $(M,g,\mathbb{Q})$ (or $(M,g,\mathbb{Q},\eta)$).
\end{dfn}
Note that given a QC structure generates a $2$-sphere bundle $Q$ of almost complex structures on $H$, locally given by $Q=\{aI_1+bI_2+cI_3|a^2+b^2+c^2=1\}$.
As Biquard shows in \cite{Biq1}, given a contact form $\eta$ on $M$ determines in a unique way the metric and the quaternionic structure on the horizontal space $H$ (if they exist). Moreover, the rotation of the contact form and the quaternionic structure (i.e. the almost complex structures $I_1, I_2$ and $I_3$) by the same rotation gives again a contact form and an almost complex structures, satisfying the above conditions (the metric is unchanged). Another essential fact is that given a horizontal distribution and a metric on it determine at most one $2$-sphere bundle of associated contact forms and a corresponding $2$-sphere bundle of almost complex structures \cite{Biq1}.

Basic (and essential) examples of QC manifolds are the quaternionic Heisenberg group $\mathbf{G}(\mathbb{H})$ (the flat model), endowed with the corresponding QC structure, and  the $3$-Sasakian manifolds, see \cite{IV2}.

On a quaternionic contact manifold with a fixed horizontal distribution $H$ and a metric $g$ on it there exists a canonical connection, the \emph{Biquard connection}, defined in \cite{Biq1}. Precisely, the next theorem holds.
\begin{thrm}\label{Biq}[O. Biquard, \cite{Biq1}] Let $(M,g,\mathbb{Q})$ be a QC manifold of dimension $4n+1>7$ with a fixed horizontal distribution $H$ and a metric $g$ on it. Then there exist a unique connection $\nabla$ on $M$ with torsion tensor $T$ and a unique supplementary distribution $V$ to $H$ in $TM$, such that the next conditions hold:
\begin{enumerate}
\item $\nabla$ preserves the decomposition $H\oplus V$ and the $Sp(n)Sp(1)$-structure on $H$, i.e. $\nabla g=0$ and $\nabla\sigma\in\Gamma(\mathbb{Q})$ for any section $\sigma\in\Gamma(\mathbb{Q});$
\item the restriction of the torsion on $H$ is given by $T(X,Y)=-[X,Y]_{|V}$ and for any vector field $\xi\in\Gamma(V)$ the torsion endomorphism $T_\xi(\cdot):=T(\xi,\cdot)_{|H}$ of $H$  lies in $(sp(n)\oplus sp(1))^\perp\subset gl(4n)$;
\item the connection on $V$ is generated by the natural identification $\varphi$ of $V$ with the subspace $sp(1):=span\{I_1,I_2,I_3\}$ of the endomorphisms on $H$, or in other words, $\nabla\varphi=0.$
\end{enumerate}
\end{thrm}
Throughout this paper we shall denote by $\nabla$ only the Biquard connection.
Note that in the condition $(2)$ of Theorem~\ref{Biq} the inner product $<\cdot,\cdot>$ of the endomorphisms on $H$ is given by 
$$<\Phi,\Psi>:=\sum_{a=1}^{4n}g(\Phi(e_a),\Psi(e_a)),\quad \Phi, \Psi\in End(H).$$ 
In \cite{Biq1} Biquard explicitly describes the supplementary subspace $V$ (the \emph{vertical space}) on the QC-manifolds of dimension bigger than seven. Namely, $V$ is locally generated by the three vector fields $\xi_1,\xi_2$ and $\xi_3$ (called \emph{Reeb vector fields}), i.e. $V=span\{\xi_1,\xi_2,\xi_3\},$ satisfying the conditions:
\begin{equation}\label{Reeb}
\eta_s(\xi_t)=\delta_{st},\qquad (\xi_s\lrcorner d\eta_t)_{|H}=-(\xi_t\lrcorner d\eta_s)_{|H},\qquad (\xi_s\lrcorner d\eta_s)_{|H}=0,
\end{equation}
where $\lrcorner$ means the interior multiplication of a vector field and a differential form. 

In the seven dimensional case the Biquard's theorem is not always true. However, Duchemin \cite{D} shows that if we assume the existence of the Reeb vector fields, satisfying the conditions \eqref{Reeb}, then Theorem~\ref{Biq} holds. Because of this, throughout this paper we shall assume that a QC structure in the 7D case satisfies the conditions \eqref{Reeb}. 

The Riemannian metric $g$ on $H$ can be extended to a metric on the entire $TM$ (i.e. to a Riemannian  metric on $M$) by the requirements $H\perp V$ and $g(\xi_s,\xi_t)=\delta_{st}.$ Note that the extended metric (which we shall again denote by $g$) is invariant under the rotations in $V,$ i.e. the action of the group $SO(3)$ on $V,$ and of course is parallel with respect to $\nabla$, $\nabla g=0.$

The \emph{fundamental $2$-forms} $\omega_s$ of the quaternionic structure $(\mathbb{Q},g)$ on $H$ are defined in a standard way by $$\omega_s(X,Y):=g(I_sX,Y),\quad s=1,2,3,$$ and can be extended to $2$-forms on $M$ by requirement $\xi\lrcorner\omega_s=0,\quad \xi\in\Gamma(V).$

The covariant derivatives of the quaternionic structure  and the Reeb vector fields with respect to the Biquard connection are given by
\begin{equation}\label{xider}
\nabla I_i=-\alpha_j\otimes I_k+\alpha_k\otimes I_j,\quad
\nabla\xi_i=-\alpha_j\otimes\xi_k+\alpha_k\otimes\xi_j,
\end{equation}
where $\alpha_s, s=1,2,3,$ are the $sp(1)$-connection $1$-forms of the Biquard connection.

The orthonormal frame $\{e_1,e_2=I_1e_1,e_3=I_2e_1,e_4=I_3e_1,\ldots,e_{4n}=I_3e_{4n-3},\xi_1,\xi_2,\xi_3\}$ of $TM$ is called a \emph{QC-normal frame} at a given point $p\in M$, if the connection $1$-forms of the Biquard connection vanish at $p$.  The existence of a QC-normal frame at any point of $M$ is provided by Lemma 4.5. in \cite{IMV}.
\subsection{Invariant decompositions of the endomorphisms of $H$}\label{decomp}
Any endomorphism $\Psi:H\rightarrow H$ of $H$ can be decomposed in a unique way into four $Sp(n)$-invariant parts with respect to the quaternionic structure $(\mathbb{Q},g)$ as follows:$$\Psi=\Psi^{+++}+\Psi^{+--}+\Psi^{-+-}+\Psi^{--+},$$ 
where $\Psi^{+++}$ commutes with all three $I_i$, $\Psi^{+--}$ commutes with $I_1$ and anti-commutes with the others two, etc. Further, we can regard $\Psi$ as decomposed into two $Sp(n)Sp(1)$-invariant parts with respect to $(\mathbb{Q},g)$, $\Psi=\Psi_{[3]}+\Psi_{[-1]},$ where $\Psi_{[3]}=\Psi^{+++}, \Psi_{[-1]}=\Psi^{+--}+\Psi^{-+-}+\Psi^{--+}.$ Note that in the above decomposition the lower indices $[3]$ and $[-1]$ arise from the fact that $\Psi_{[3]}$ and $\Psi_{[-1]}$ appear the projections of $\Psi$ on the eigenspaces of the Casimir operator
\begin{equation*}
\Upsilon =\ I_1\otimes I_1\ +\ I_2\otimes I_2\ +\ I_3\otimes I_3,
\end{equation*}
corresponding, respectively, to the eigenvalues $3$ and $-1$, see \cite{CSal}. 

In the case $n=1$ an important fact is that the space of the symmetric endomorphisms of $H$, commuting with all three almost complex structures $I_s$, is one-dimensional. Consequently, the $[3]$-component $\Psi_{[3]}$ of any symmetric endomorphism $\Psi$ of $H$ is proportional to the identity operator $Id_{|H}$ of $H,$  explicitly, $\Psi_{[3]}=-\frac{tr\Psi}{4}Id_{|H}.$
\subsection{The torsion and the curvature of Biquard connection}
The torsion tensor $T$ of Biquard connection is defined as usually by $$T(A,B)=\nabla_AB-\nabla_BA-[A,B].$$ The corresponding tensor of type $(0,3)$ via the metric $g$ is obtained in a standard way and is denoted by the same letter, $T(A,B,C)=g(T(A,B),C).$ The restriction of the torsion to the horizontal space  $H$ has the expression $$T(X,Y)=-[X,Y]_{|V}=2\sum_{s=1}^3\omega_s(X,Y)\xi_s,$$ see \cite{IV2}. For an arbitrary but fixed vertical vector field $\xi\in\Gamma(V)$ one obtains an endomorphism $T_\xi$ on $H,$
defined by
$$T_\xi(\cdot):=T(\xi,\cdot)_{|H}:H\rightarrow H.$$ The torsion endomorphism $T_\xi$ is completely trace-free \cite{Biq1}, i.e. $trT_\xi=tr(T_\xi\circ I_s)=0,$ or explicitly 
\begin{equation}\label{trtor}
T(\xi,e_a,e_a)=T(\xi,e_a,I_se_a)=0.
\end{equation}
We shall need the identities
\begin{equation}\label{trtor1}
T(\xi_i,\xi_k,\xi_i)=T(\xi_i,\xi_j,\xi_i)=0,
\end{equation}
see e.g. \cite[(4.34)]{IV2}.
The torsion endomorphism $T_{\xi}$ can be decomposed in a standard way into a symmetric $T_\xi^0$ and a skew-symmetric  $b_\xi$ parts, $T_\xi=T_\xi^0+b_\xi,$ and the symmetric part enjoys the properties
\begin{equation}
\begin{aligned}
T_{\xi _{i}}^{0}I_{i}=-I_{i}T_{\xi _{i}}^{0},\quad I_{2}(T_{\xi
_{2}}^{0})^{+--}=I_{1}(T_{\xi _{1}}^{0})^{-+-},\\\quad I_{3}(T_{\xi
_{3}}^{0})^{-+-}=I_{2}(T_{\xi _{2}}^{0})^{--+},\quad I_{1}(T_{\xi
_{1}}^{0})^{--+}=I_{3}(T_{\xi _{3}}^{0})^{+--}.
\end{aligned}
\end{equation}
For a fixed Reeb vector field $\xi_i$ the skew-symmetric part $b_{\xi_{i}}$ of $T_{\xi_{i}}$ can be represented as $b_{\xi_{i}}=I_iU,$ where $U$ is a traceless symmetric endomorphism of $H,$ which commutes with all three almost complex structures $I_s,\quad s=1,2,3.$ As a consequence in the case $n=1$ one obtains that the  tensor $U$ vanishes identically, $U=0,$ (see the end of Subsection~\ref{decomp}) and the torsion endomorphism $T_\xi$ is a symmetric tensor, $T_\xi=T_\xi^0.$

Ivanov et al. have  introduced \cite{IMV} the two $Sp(n)Sp(1)$-invariant symmetric and traceless tensors $T^0$ and $U$ on $H,$ defined by 
\begin{equation}  \label{Tcompnts}
T^0(X,Y)= g((T_{\xi_1}^{0}I_1+T_{\xi_2}^{0}I_2+T_{ \xi_3}^{0}I_3)X,Y) \
\text{ and }\ U(X,Y) =g(UX,Y).
\end{equation}
These tensors satisfy the equalities
\begin{equation}  \label{propt}
\begin{aligned} T^0(X,Y)+T^0(I_1X,I_1Y)+T^0(I_2X,I_2Y)+T^0(I_3X,I_3Y)=0, \\
U(X,Y)=U(I_1X,I_1Y)=U(I_2X,I_2Y)=U(I_3X,I_3Y). \end{aligned}
\end{equation}
The symmetric part $T_{\xi_s}^0$ of $T_{\xi_s}$ enjoys the property \cite[Proposition~2.3]{IV}
\begin{equation}  \label{need}
4T^0(\xi_s,I_sX,Y)=T^0(X,Y)-T^0(I_sX,I_sY),
\end{equation}
where as usually $T^0(\xi,X,Y)=g(T^0(\xi,X),Y) \Big(=g(T_{\xi}^0(X),Y)\Big).$
As a corollary of \eqref{propt} and \eqref{need} we obtain the equality
\begin{multline}  \label{need1}
T(\xi_s,I_sX,Y)=T^0(\xi_s,I_sX,Y)+g(I_sUI_sX,Y) \\
=\frac14\Big[T^0(X,Y)-T^0(I_sX,I_sY)\Big]-U(X,Y).
\end{multline}
As a consequence of \eqref{propt} and \eqref{need1} we get
\begin{equation}\label{tornab}
\sum_{s=1}^3T(\xi_s,I_sX,Y)=T^0(X,Y)-3U(X,Y).
\end{equation}

The curvature tensor $R$ of Biquard connection is defined in a standard way by 
$$R(A,B,C)=\nabla_A\nabla_BC-\nabla_B\nabla_AC-\nabla_{[A,B]}C.$$
The corresponding tensor of type $(0,4)$ with respect to the metric $g$ is denoted by the same letter,
$R(A,B,C,D):=g(R(A,B,C),D).$ 

There are several tensors, arising from the curvature tensor, which play crucial role in the QC geometry. {The
\emph{QC-Ricci tensor} $Ric$, the \emph{QC-scalar curvature} $Scal$, the \emph{normalized QC-scalar curvature} $S$, the \emph{QC-Ricci forms} $\rho_s$ and the \emph{Ricci-type tensors} $\zeta_s$ of the
Biquard connection are defined, respectively, by the following formulas.
\begin{equation}  \label{qscs}
\begin{aligned} & Ric(A,B)=R(e_b,A,B,e_b),\quad Scal=R(e_b,e_a,e_a,e_b),\quad
8n(n+2)S=Scal,\\ & \rho_s(A,B)=\frac1{4n}R(A,B,e_a,I_se_a),
\quad \zeta_s(A,B)=\frac1{4n}R(e_a,A,B,I_se_a). \end{aligned}
\end{equation}
}
Some significant relations between the upper objects and the torsion tensors are established in \cite{IMV} (see also \cite{IMV1,IV}). Namely, the next formulas hold.
\begin{equation}  \label{sixtyfour}
\begin{aligned} & Ric(X,Y) =(2n+2)T^0(X,Y)+(4n+10)U(X,Y)+2(n+2)Sg(X,Y),\\
&\zeta_s(X,I_sY)=\frac{2n+1}{4n}T^0(X,Y)+\frac1{4n}T^0(I_sX,I_sY)+%
\frac{2n+1}{2n}U(X,Y)+\frac{S}2g(X,Y), \\ &
T(\xi_{i},\xi_{j})=-S\xi_{k}-[\xi_{i},\xi_{j}]_{|H}, \qquad S =
-g(T(\xi_1,\xi_2),\xi_3),\\ &
g(T(\xi_i,\xi_j),X)=-\rho_k(I_iX,\xi_i)=-\rho_k(I_jX,\xi_j)=-g([\xi_i,%
\xi_j],X). \end{aligned}
\end{equation}
In the seven dimensional case ($n=1$) the upper formulas are valid with $U=0.$

An important class of QC structures consists of the \emph{QC-Einstein structures}, defined as follows.
\begin{dfn}A QC structure is called QC-Einstein, if the horizontal restriction of the QC-Ricci tensor is proportional to the metric, i.e.  
\begin{equation}\label{qcein}
Ric(X,Y)=2(n+2)Sg(X,Y).
\end{equation}
\end{dfn}
A manifold endowed with a QC-Einstein structure is called \emph{QC-Einstein manifold}. The first equality in \eqref{sixtyfour} implies that the QC-Einstein condition (the equation \eqref{qcein}) is equivalent to the vanishing of the torsion endomorphism, i.e. $T^0=U=0.$ An established in \cite{IMV} result asserts, that a QC-Einstein structure of dimension bigger than seven has constant QC-scalar curvature, and the vertical distribution is integrable. The corresponding result in the seven-dimensional case is recently established in \cite{IMV3}.

Note that the vanishing of the horizontal restriction of the $sp(n)$-connection $1$-forms $\alpha_s, s=1,2,3,$ implies the vanishing of the torsion endomorphism $T_\xi$ of the Biquard connection, see \cite{IMV}.

Examples of QC-Einstein manifolds are the $3$-Sasakian manifolds, since they have zero torsion endomorphism. The converse is also true in a local sense, namely, any QC-Einstein manifold with positive QC-scalar curvature is locally $3$-Sasakian \cite{IMV} (see \cite{IV1} for the case of negative QC-scalar curvature).
\subsection{The horizontal divergence theorem and the sub-Laplacian}
On a QC manifold $(M,g,\mathbb{Q})$ of dimension $4n+3$ the \emph{horizontal divergence} of a horizontal $1$-form (or a horizontal vector field) $\omega\in\Lambda^1(H)$ is defined by $$\nabla^*\omega=-tr|_{H}\nabla\omega=-\nabla\omega(e_a,e_a).$$ 
If $\eta=(\eta_1,\eta_2,\eta_3)$ is a fixed local contact form of the QC manifold then for an arbitrary  $s\in\{1,2,3\}$ the form $Vol_\eta=\eta_1\wedge\eta_2\wedge\eta_3\wedge\omega_s^{2n}$ is locally defined volume form, which is independent of the choice of $s$ and the local $1$-forms $\eta_1,\eta_2$ and $\eta_3.$ Consequently, $Vol_\eta$ is globally defined volume form on $(M,g,\mathbb{Q}).$ If the QC manifold is compact, the integration by parts over $M$ is possible due to the next divergence formula:
\begin{equation*}  \label{div}
\int_M (\nabla^*\omega)\,\, Vol_{\eta}\ =\ 0,
\end{equation*}
see \cite{IMV}, \cite{Wei}.

For a smooth function $f$ on $M$ the \emph{horizontal Hessian} $\nabla^2f(\cdot,\cdot):\Gamma(H)\times\Gamma(H)\rightarrow\Lambda^0(M)$ and the \emph{sub-Laplacian} $\Delta f\in\Lambda^0(M)$ are defined in a standard way  by $$\nabla^2f(X,Y)=(\nabla_Xdf)(Y)\quad\textit{and}\quad\Delta f=\nabla^*df=-\nabla^2f(e_a,e_a).$$ By definition, the \emph{horizontal gradient} of $f$ is the vector field $\nabla f,$ s.t. $$g(\nabla f,X)=df(X),\quad X\in\Gamma (H).$$

Any (non-zero) smooth function $f$ satisfying the equation $\Delta f=\lambda f$ for some constant $\lambda$ is called \emph{eigenfunction}, corresponding to the \emph{eigenvalue} $\lambda$ of $\Delta.$ In the case of compact $M$ the last equation and the divergence formula yield the non-negativity of the spectrum of the sub-Laplacian.
\section{Some basic identities}
In this section we list some identities which we shall use in the proof of the main results. 
We shall need the next \emph{Ricci identities} \cite{IMV,IV2}
\begin{equation}  \label{boh2}
\begin{aligned} & \nabla^2f
(X,Y)-\nabla^2f(Y,X)=-2\sum_{s=1}^3\omega_s(X,Y)df(\xi_s),\\ & \nabla^2f
(X,\xi_s)-\nabla^2f(\xi_s,X)=T(\xi_s,X,\nabla f),\\ & \nabla ^{3}f(\xi _{i},X,Y) =\nabla
^{3}f(X,Y,\xi _{i})-\nabla ^{2}f\left( T\left( \xi _{i},X\right) ,Y\right)
-\nabla ^{2}f\left( X,T\left( \xi _{i},Y\right) \right) \\
&\hskip3in-df\left( \left( \nabla _{X}T\right) \left( \xi _{i},Y\right)
\right) -R(\xi _{i},X,Y,\nabla f) . \end{aligned}
\end{equation}
As a consequence of the first identity in \eqref{boh2} we get 
\begin{equation}  \label{xi1}
g(\nabla^2f , \omega_s) =\nabla^2f(e_a,I_se_a)=-4ndf(\xi_s).
\end{equation}

The next basic formula that we shall use is a representation of the curvature tensor \cite{IV,IV2}
\begin{multline}  \label{d3n5}
R(\xi _{i},X,Y,Z)=-(\nabla _{X}U)(I_{i}Y,Z)+\omega _{j}(X,Y)\rho
_{k}(I_{i}Z,\xi _{i})
-\omega _{k}(X,Y)\rho _{j}(I_{i}Z,\xi _{i})\\-\omega _{j}(X,Z)\rho _{k}(I_{i}Y,\xi _{i})
+\omega _{k}(X,Z)\rho _{j}(I_{i}Y,\xi _{i})-\omega _{j}(Y,Z)\rho
_{k}(I_{i}X,\xi _{i}) +\omega _{k}(Y,Z)\rho _{j}(I_{i}X,\xi _{i})\\-\frac{1}{4}\Big[(\nabla
_{Y}T^{0})(I_{i}Z,X)+(\nabla _{Y}T^{0})(Z,I_{i}X)\Big]
+\frac{1}{4}\Big[(\nabla _{Z}T^{0})(I_{i}Y,X)+(\nabla _{Z}T^{0})(Y,I_{i}X)%
\Big],
\end{multline}%
where the \emph{Ricci $2$-forms} are given by (see \cite{IV} or \cite{IV2})
\begin{equation}\label{d3n6}
\begin{aligned}
&6(2n+1)\rho_s(\xi_s,X)=(2n+1)X(S)+\frac12
(\nabla_{e_a}T^0)[(e_a,X)-3(I_se_a,I_sX)]\\&\hskip4.3in
-2(\nabla_{e_a}U)(e_a,X),\\ 
&6(2n+1)\rho_i(\xi_j,I_kX)=-6(2n+1)\rho_i(\xi_k,I_jX)=(2n-1)(2n+1)X(S)\\&\hskip0.5in-\frac{4n+1}{2}(
\nabla_{e_a}T^0)(e_a,X)-\frac{3}{2}(\nabla_{e_a}T^0)(I_ie_a,I_iX)-4(n+1)(\nabla_{e_a}U)(e_a,X)
.
\end{aligned}
\end{equation}

By the well-known formula for the relation between two metric connections, we obtain the next one in  the case of the Biquard connection $\nabla$ and the Levi-Civita connection $\nabla^g$ of the extended Riemannian metric $g$:
\begin{equation}\label{diferfor}
g(\nabla_AB,C)-g(\nabla_A^gB,C)=\frac12\Big(T(A,B,C)-T(B,C,A)+T(C,A,B)\Big).
\end{equation}

\section{Proof of Theorem~\ref{Main}}
Let $\lambda$ is the first (non-zero) eigenvalue of the sub-Laplacian and $f$ is a smooth function on $M$ that satisfies the equalities 
\begin{equation}\label{eigenf}
\Delta f=\lambda f\quad\textit{and}\quad\int_Mf^2\,Vol_\eta=1.
\end{equation}
Note that the second equality in \eqref{eigenf} can be always obtained by a suitable constant rescaling of $f$. The proof of Theorem~\ref{Main} lies on a sequence of lemmas, which we shall formulate and prove here. We start with the next
\begin{lemma}\label{lemma1}Let $(M,g,\mathbb{Q})$ be a compact quaternionic contact manifold of dimension seven. Then the next integral inequality holds true
\begin{equation}\label{mainineq}
\int_M\Big[Ric(\nabla f,\nabla f)-2T^0(\nabla f,\nabla f)-\frac{3}{4}\lambda|\nabla f|^2-12\sum_{s=1}^3\Big(df(\xi_s)\Big)^2\Big]Vol_\eta\leq 0.
\end{equation}
\end{lemma}
\begin{proof}
Following \cite{LL}, we begin with the Bochner-type formula, established in our previous paper \cite[(3.3)]{IPV1}
\begin{multline}\label{bohh}
-\frac12\triangle |\nabla f|^2=|\bi^2f|^2-g\left (\nabla (\triangle f), \gr \right )+Ric(\nabla f,\nabla f)+2\sum_{s=1}^3T(\xi_s,I_s\gr,\gr)\\+
4\sum_{s=1}^3\bi^2f(\xi_s,I_s\gr).
\end{multline}
As well as in the higher dimensions, this formula is at the root of the proof of the desired estimate. The next basic formula is \cite[(3.3)]{IPV3}
\begin{equation}  \label{e:gr3}
\sum_{s=1}^3\nabla^{2}f(\xi_s,I_sX)=\frac{1}{4n}\sum_{s=1}^3%
\nabla^{3}f(I_sX,I_se_a,e_a)-\sum_{s=1}^3T(\xi_s,I_sX,\nabla f).
\end{equation}
Integrating over $M$ the both sides of \eqref{e:gr3} for $n=1$ and $X=\nabla f$ and using the integral identity
\begin{equation}\label{intid}
\int_M\sum_{s=1}^3\nabla^3f(I_s\nabla f,I_se_a,e_a)\,Vol_\eta=-16\int_M\sum_{s=1}^3\Big(df(\xi_s)\Big)^2\,Vol_\eta
\end{equation}
and \eqref{tornab}, we obtain
\begin{equation}\label{inteq}
\int_M\sum_{s=1}^3\nabla^2f(\xi_s,I_s\nabla f)\,Vol_\eta=-\int_M\Big[4\sum_{s=1}^3\Big(df(\xi_s)\Big)^2+T^0(\nabla f,\nabla f)\Big]\,Vol_\eta.
\end{equation} 
It should be noted that in the calculations for getting \eqref{intid} we used \eqref{xi1}, an integration by parts and the $Sp(n)Sp(1)-$invariance of the expression  $\sum_{s=1}^3\nabla^3f(I_s\nabla f,I_se_a,e_a),$ which allows us to work in a QC-normal frame.

Further, we take the next inequalities for the $Sp(n)Sp(1)$-invariant parts of the horizontal Hessian, \cite[(4.6) and (4.7)]{IPV1}, 
\begin{equation*}
|(\bi^2f)_{[-1]}|^2 \geq 4n\sum_{s=1}^3\Big(df(\xi_s)\Big)^2,\quad|(\bi^2f)_{[3]}|^2 \geq \frac1{4n}(\triangle f)^2,
\end{equation*}
which in the seven-dimensional case $(n=1)$ give the next inequality for the norm of the horizontal Hessian:
\begin{equation}\label{hnorm}
|\nabla^2f|^2=|(\nabla^2f)_{[-1]}|^2+|(\nabla^2f)_{[3]}|^2\geq 4\sum_{s=1}^3\Big(df(\xi_s)\Big)^2+\frac14(\Delta f)^2.
\end{equation}

Taking into account the divergence formula, we get the integral identity
\begin{equation}  \label{eq444}
\int_M (\triangle f)^2\, Vol_{\eta}=\lambda \int_M |\nabla f|^2 \, Vol_{\eta}.
\end{equation}

Finally, integrating \eqref{bohh} over $M$ and using \eqref{tornab}, \eqref{inteq}, \eqref{hnorm} and  \eqref{eq444}, we obtain \eqref{mainineq}.
\end{proof}
 The next decisive step is to find a convenient estimate of the term $\int_M\sum_{s=1}^3\Big(df(\xi_s)\Big)^2\,Vol_{\eta},$ which appears in \eqref{mainineq}. The aim of the following results is to establish one such estimate.
\begin{lemma}["Vertical Bochner formula"]\label{lema2} Let $\phi$ be a smooth function on a QC manifold  $(M,g,\mathbb{Q})$ of dimension $4n+3.$ Then the next formula holds
\begin{multline}\label{vertBoch}
\sum_{s=1}^3\Delta(\xi_s\phi)^2=2\sum_{s=1}^3\Big[-|\nabla(\xi_s\phi)|^2+d\phi(\xi_s)\xi_s(\Delta \phi)-d\phi(\xi_s)R(\xi_s,e_a,e_a,\nabla \phi)\\-d\phi(\xi_s)(\nabla_{e_a}T)(\xi_s,e_a,\nabla \phi)-2d\phi(\xi_s)g\Big(T_{\xi_s},\nabla^2\phi\Big)\Big].
\end{multline}
\end{lemma}
\begin{proof}
First, it should be noted, that the tensor $T_{\xi_s}$ that appears in the last term of the right-hand side of \eqref{vertBoch}, is assumed to be the tensor of type $(0,2),$ corresponding to the torsion endmorphism $T_{\xi_s}$ via $g.$ The left hand side of the desired equality \eqref{vertBoch} is an $Sp(n)Sp(1)$-invariant and hence we can do our computations in a QC-normal frame. Using the first and the third Ricci identity in \eqref{boh2} and the properties of the torsion endomorphism, after some standard calculations we have
\begin{multline}
\sum_{s=1}^3\Delta(\xi_s\phi)^2\\=2\sum_{s=1}^3\Big[-|\nabla(\xi_s\phi)|^2+d\phi(\xi_s)\Delta(\xi_s\phi)\Big]=2\sum_{s=1}^3\Big[-|\nabla(\xi_s\phi)|^2-d\phi(\xi_s)\nabla^3\phi(e_a,e_a,\xi_s)\Big]\\=2\sum_{s=1}^3\Big[-|\nabla(\xi_s\phi)|^2-d\phi(\xi_s)\Big(\nabla^3\phi(\xi_s,e_a,e_a)+\nabla^2\phi(T(\xi_s,e_a),e_a)+\nabla^2\phi(e_a,T(\xi_s,e_a))\\+d\phi((\nabla_{e_a}T)(\xi_s,e_a))+R(\xi_s,e_a,e_a,\nabla\phi)\Big)\Big]=2\sum_{s=1}^3\Big[-|\nabla(\xi_s\phi)|^2+d\phi(\xi_s)\xi_s(\Delta \phi)\\-d\phi(\xi_s)R(\xi_s,e_a,e_a,\nabla \phi)-d\phi(\xi_s)(\nabla_{e_a}T)(\xi_s,e_a,\nabla \phi)-2d\phi(\xi_s)g\Big(T_{\xi_s},\nabla^2\phi\Big)\Big],
\end{multline}
which proofs the lemma.
\end{proof}
Applying \eqref{vertBoch} to the case of a seven-dimensional QC manifold and an eigenfunction $f$ on it, we obtain the next
\begin{lemma}\label{lema3} On a QC manifold $(M,g,\mathbb{Q})$ of dimension seven the next formula holds
\begin{multline}\label{sublap}
\sum_{s=1}^3\Delta(\xi_sf)^2=2\sum_{s=1}^3\Big[-|\nabla(\xi_sf)|^2+\lambda\Big(df(\xi_s)\Big)^2-\frac{2}{3}df(\xi_s)dS(I_s\nabla f)\\+\frac{8}{9}df(\xi_s)(\nabla_{e_a}T^0)(e_a,I_s\nabla f)-\frac{2}{3}df(\xi_s)\Big((\nabla_{e_a}T^0)(\xi_u,e_a,I_t\nabla f)-(\nabla_{e_a}T^0)(\xi_t,e_a,I_u\nabla f)\Big)\\-2df(\xi_s)e_a\Big(T(\xi_s,e_a,\nabla f)\Big)\Big].
\end{multline}
\end{lemma}
\begin{proof}
As in the proof of the previous lemma, we can work in a QC-normal frame. Using the properties of the torsion tensor, listed in Subsection 2.3., we get 
\begin{multline}\label{nablator}
\sum_{s=1}^3df(\xi_s)(\nabla_{e_a}T)(\xi_s,e_a,\nabla f)\\=-\frac{1}{4}\sum_{s=1}^3df(\xi_s)\Big[(\nabla_{e_a}T^0)(\nabla f,I_se_a)+(\nabla_{e_a}T^0)(I_s\nabla f,e_a)\Big].
\end{multline}
Next we use \eqref{d3n5} and the properties of the torsion tensor to obtain 
\begin{multline}\label{curvf}
\sum_{s=1}^3df(\xi_s)R(\xi_s,e_a,e_a,\nabla f)\\=\sum_{s=1}^3df(\xi_s)\Big[-\frac{1}{4}\Big((\nabla_{e_a}T^0)(I_s\nabla f,e_a)+(\nabla_{e_a}T^0)(\nabla f,I_se_a)\Big)\\-2\omega_t(e_a,\nabla f)\rho_u(I_se_a,\xi_s)+2\omega_u(e_a,\nabla f)\rho_t(I_se_a,\xi_s)\Big].
\end{multline}
We use the representations \eqref{d3n6} for the Ricci $2$-forms that appear in \eqref{curvf} to obtain
\begin{equation}\label{ricf}
\begin{aligned}
&\rho_u(I_se_a,\xi_s)=-\frac16dS(I_ue_a)+\frac{5}{36}(\nabla_{e_{b}}T^0)(e_b,I_ue_a)-\frac{1}{12}(\nabla_{e_b}T^0)(I_ue_b,e_a),\\
&\rho_t(I_se_a,\xi_s)=-\frac16dS(I_te_a)+\frac{5}{36}(\nabla_{e_b}T^0)(e_b,I_te_a)-\frac{1}{12}(\nabla_{e_b}T^0)(I_te_b,e_a).
\end{aligned}
\end{equation}
Substituting \eqref{nablator}, \eqref{curvf} and \eqref{ricf} in the right-hand side of \eqref{vertBoch} and using the properties of the torsion tensor, we get \eqref{sublap} after a number of standard computations. 
\end{proof}
The next Lemma gives an integral equality, which is one of the main instruments for obtaining the needed sharp estimate for the term $\int_M\sum_{s=1}^3\Big(df(\xi_s)\Big)^2\,Vol_{\eta},$ that occurs in \eqref{mainineq}.
\begin{lemma}\label{lema4}
On a seven-dimensional compact QC manifold $(M,g,\mathbb{Q})$ the next integral formula holds
\begin{multline}\label{mainint}
\int_M\sum_{s=1}^3|\nabla(\xi_sf)|^2\,Vol_{\eta}=\int_M\sum_{s=1}^3\Big[2|T(\xi_s,\nabla f)|^2+\frac{1}{6}(I_s\nabla f)^2S\\-\frac{2}{9}I_s\nabla f\Big((\nabla_{e_a}T^0)(e_a,I_s\nabla f)\Big)+\frac{1}{6}I_s\nabla f\Big((\nabla_{e_a}T)(\xi_u,e_a,I_t\nabla f)\Big)\\-\frac{1}{6}I_s\nabla f\Big((\nabla_{e_a}T)(\xi_t,e_a,I_u\nabla f)\Big)-(\nabla_{\xi_s}T)(\xi_s,\nabla f,\nabla f)+\lambda\Big(df(\xi_s)\Big)^2\Big]\,Vol_{\eta}.
\end{multline} 
\end{lemma}
\begin{proof}
We begin with integrating over $M$ the both sides of \eqref{sublap}. We shall work as before in a QC-normal frame in view of the $Sp(n)Sp(1)$-invariance of the considerating tensors. Having in mind the divergence formula, we shall simplify some of the terms that appear under the integral.

Using \eqref{xi1} and an integration by parts, after some standard calculations we get the next identities
\begin{equation}\label{firstt}
\int_M\sum_{s=1}^3df(\xi_s)dS(I_s\nabla f)\,Vol_{\eta}=-\frac14\int_M\sum_{s=1}^3(I_s\nabla f)^2S\,Vol_{\eta},
\end{equation}
\begin{equation}\label{secondt}
\int_M\sum_{s=1}^3df(\xi_s)(\nabla_{e_a}T^0)(e_a,I_s\nabla f)\,Vol_\eta=-\frac14\int_M\sum_{s=1}^3I_s\nabla f\Big((\nabla_{e_a}T^0)(e_a,I_s\nabla f)\Big)\,Vol_{\eta},
\end{equation}
\begin{multline}\label{thirdt}
\int_M\sum_{s=1}^3df(\xi_s)(\nabla_{e_a}T)(\xi_u,e_a,I_t\nabla f)\,Vol_{\eta}\\=-\frac14\int_M\sum_{s=1}^3I_s\nabla f\Big((\nabla_{e_a}T)(\xi_u,e_a,I_t\nabla f)\Big)\,Vol_{\eta}.
\end{multline}

In order to transform the term $\int_M\sum_{s=1}^3df(\xi_s)e_a\Big(T(\xi_s,e_a,\nabla f)\Big)\,Vol_{\eta},$ let us give some auxiliary notions and facts. We shall denote by $div^{\nabla}$ and $div^{{\nabla}^g}$ the divergences corresponding to the Biquard connection $\nabla$ and to the Levi-Civita connection $\nabla^g,$ respectively. For any vertical vector field $\xi$ on a QC manifold of dimension $4n+3$ we have 
\begin{multline}\label{divrel}
div^{{\nabla}^g}(\xi)=\sum_{a=1}^{4n}g(\nabla_{e_a}^g\xi,e_a)+\sum_{s=1}^3g(\nabla_{\xi_s}^g\xi,\xi_s)=\sum_{a=1}^{4n}g(\nabla_{e_a}\xi,e_a)+\sum_{s=1}^3g(\nabla_{\xi_s}\xi,\xi_s)\\=div^{\nabla}(\xi),
\end{multline}
where for the second equality we used \eqref{diferfor} and the properties of the torsion tensor \eqref{trtor} and \eqref{trtor1}. Because of the volume form $Vol_{\eta}$ differs from the Riemannian volume form $d\mu^g$ by a constant multiplier $C$, $Vol_\eta=C.d\mu^g,$ we get by the Riemannian divergence formula and \eqref{divrel} 
\begin{equation}\label{twodiv}
\int_M div^{\nabla}(\xi)\,Vol_{\eta}=C\int_M div^{\nabla}(\xi)\,d\mu^g=C\int_M div^{{\nabla}^g}(\xi)\,d\mu^g=0.
\end{equation}
We have consecutively
\begin{multline}\label{lel}
\int_M\sum_{s=1}^3df(\xi_s)e_a\Big(T(\xi_s,e_a,\nabla f)\Big)\,Vol_{\eta}=-\int_M\sum_{s=1}^3\nabla^2f(e_a,\xi_s)T(\xi_s,e_a,\nabla f)\,Vol_\eta\\=-\int_M\sum_{s=1}^3\Big[T(\xi_s,e_a,\nabla f)T(\xi_s,e_a,\nabla f)+\nabla^2f(\xi_s,e_a)T(\xi_s,e_a,\nabla f)\Big]\,Vol_{\eta}\\=-\int_M\sum_{s=1}^3\Big[|T(\xi_s,\nabla f)|^2-df(e_a)\xi_s\Big(T(\xi_s,e_a,\nabla f)\Big)\Big]\,Vol_\eta\\=\int_M\sum_{s=1}^3\Big[-|T(\xi_s,\nabla f)|^2+\frac{1}{2}(\nabla_{\xi_s}T)(\xi_s,\nabla f,\nabla f)\Big]\,Vol_\eta,
\end{multline}
where we used an integration by parts for the first equality of the above chain, next we took account the second Ricci identity in \eqref{boh2} to obtain the second one, and last, in order to get the third and the fourth equalities  we used \eqref{twodiv}  for the vertical vector field $\xi:=T(\xi_s,\nabla f,\nabla f)\xi_s.$

Now, substituting \eqref{firstt}, \eqref{secondt}, \eqref{thirdt} and \eqref{lel} in the integrated over $M$ equality \eqref{sublap}, we get \eqref{mainint}.
\end{proof}

An important role for obtaining the desired estimate plays the following integral equality
\begin{equation}\label{maininteg}
\int_M\sum_{s=1}^3\Big(df(\xi_s)\Big)^2\,Vol_{\eta}=\frac{1}{4}\int_M\sum_{s=1}^3df(I_se_a)d(\xi_sf)(e_a)\,Vol_{\eta},
\end{equation}
which follows easily by \eqref{xi1} and an integration by parts.
We have the next chain of relations:
\begin{multline}\label{ch1}
\sum_{s=1}^3\int_M\lambda\Big(df(\xi_s)\Big)^2\,Vol_{\eta}=\sum_{s=1}^3\int_M\frac{\lambda}{4}df(I_se_a)d(\xi_sf)(e_a)\,Vol_{\eta}\\\leq\sum_{s=1}^3\Big[\int_M\frac{\lambda^2}{16}\Big(df(I_se_a)\Big)^2\,Vol_{\eta}\Big]^\frac12\Big[\int_M\Big(d(\xi_sf)(e_a)\Big)^2\,Vol_{\eta}\Big]^\frac12\\\leq\frac12\sum_{s=1}^3\Big[\int_M\frac{\lambda^2}{16}\Big(df(I_se_a)\Big)^2\,Vol_{\eta}+\int_M\Big(d(\xi_sf)(e_a)\Big)^2\,Vol_{\eta}\Big]\\=\frac{3\lambda^2}{32}\int_M|\nabla f|^2\,Vol_{\eta}+\frac{1}{2}\sum_{s=1}^3\int_M|\nabla(\xi_sf)|^2\,Vol_{\eta}.
\end{multline}
For the above chain we  used \eqref{maininteg} to obtain the first equality and the Cauchy-Schwarz inequality for the integral scalar product to get the first inequality. The second inequality is obtained in an obvious manner. 

Using the notation $A(X)$ from the condition of Theorem~\ref{Main}, the equality \eqref{mainint} takes the form $$\int_M\sum_{s=1}^3|\nabla(\xi_sf)|^2\,Vol_{\eta}=\int_M\Big[A(\nabla f)+\sum_{s=1}^3\lambda \Big(df(\xi_s)\Big)^2\Big]\,Vol_{\eta},$$
which, combined with \eqref{ch1}, gives the next integral inequality
\begin{equation}\label{ineqm}
\sum_{s=1}^3\int_M|\nabla(\xi_sf)|^2\,Vol_{\eta}\leq\int_M\Big[2A(\nabla f)+\frac{3\lambda^2}{16}|\nabla f|^2\Big]\,Vol_{\eta}.
\end{equation}

For any constant $b>0$ we get the next chain of relations:
\begin{multline}\label{ch2}
\sum_{s=1}^3\int_M\Big(df(\xi_s)\Big)^2\,Vol_{\eta}=\sum_{s=1}^3\int_M\frac{\sqrt{b}}{4}df(I_se_a)\frac{1}{\sqrt{b}}d(\xi_sf)(e_a)\,Vol_{\eta}\\\leq\sum_{s=1}^3\Big[\frac{b}{16}\int_M\Big(df(I_se_a)\Big)^2\,Vol_{\eta}\Big]^\frac12\Big[\frac{1}{b}\int_M\Big(d(\xi_sf)(e_a)\Big)^2\,Vol_{\eta}\Big]^\frac12\\\leq\frac{3b}{32}\int_M|\nabla f|^2\,Vol_{\eta}+\frac{1}{2b}\sum_{s=1}^{3}\int_M|\nabla(\xi_sf)|^2\,Vol_{\eta},
\end{multline}
where we used \eqref{maininteg} to obtain the equality and we took account the Cauchy-Schwarz inequality for the integral scalar product to get the first inequality. The second inequality is obvious.
Combining \eqref{ineqm} and \eqref{ch2}, we get the next key inequality
\begin{equation}\label{key}
\sum_{s=1}^3\int_M\Big(df(\xi_s)\Big)^2\,Vol_{\eta}\leq\int_M\Big[\frac{3b}{32}|\nabla f|^2+\frac{1}{b}A(\nabla f) +\frac{3\lambda^2}{32b}|\nabla f|^2\Big]\,Vol_{\eta}.
\end{equation}

Substituting \eqref{key} in \eqref{mainineq}, we obtain
\begin{equation}\label{key1}
\int_M\Big[Ric(\nabla f,\nabla f)-2T^0(\nabla f,\nabla f)-\frac{12}{b}A(\nabla f)+(-\frac34\lambda-\frac{9b}{8}-\frac{9\lambda^2}{8b})|\nabla f|^2\Big]\,Vol_{\eta}\leq 0,
\end{equation}

and if we ask the a-priori condition $$Ric(X,X)-2T^0(X,X)-\frac{12}{b}A(X)\geq k_0g(X,X)\hspace{1cm}\textit{for any}\hspace{0.5cm}X\in\Gamma(H),$$ then we have from  \eqref{key1}
\begin{equation*}
\int_M\Big(-\frac34\lambda-\frac{9b}{8}-\frac{9\lambda^2}{8b}+k_0\Big)|\nabla f|^2\,Vol_{\eta}\leq 0.
\end{equation*}
The last inequality implies 
\begin{equation*}\label{key2}
-\frac34\lambda-\frac{9b}{8}-\frac{9\lambda^2}{8b}+k_0\leq 0,
\end{equation*}
which, after the choice $b=\frac{k_0}{3},$ becomes 
\begin{equation}\label{key3}
(3\lambda-k_0)(9\lambda+5k_0)\geq 0.
\end{equation}
Since $9\lambda +5k_0>0$, the inequality \eqref{key3} gives the estimate
\begin{equation}\label{shest}
\lambda\geq\frac{k_0}{3},
\end{equation}
which ends the proof of Theorem~\ref{Main}.
\section{Proof of Corollary~\ref{corol}}
In \cite[Remark 4.1.]{IPV3} the authors give the next identity
\begin{equation}\label{cora}
10T^0(\nabla f,\nabla f)+6S|\nabla f|^2=k_0|\nabla f|^2,
\end{equation}
which holds for the extremal eigenfunction $f$ in the case of equality in Theorem~\ref{mainpan}, i.e. $\lambda=\frac13k_0$. Assuming the condition a) in Corollary~\ref{corol} and taking account \eqref{cora},  we obtain \eqref{shestS}.

In a similar way, the case of equality in Theorem~\ref{Main}, i.e. $\lambda=\frac13k_0$, together with the a-priori condition \eqref{7Dcondm-app} and \eqref{key1} give us the identity
\begin{equation*}
Ric(\nabla f,\nabla f)-2T^0(\nabla f,\nabla f)-\frac{36}{k_0}A(\nabla f)=k_0|\nabla f|^2,
\end{equation*}
which holds for the extremal eigenfunction $f$. Using the first formula in \eqref{sixtyfour}, the upper identity can be rewrited as
\begin{equation}\label{corb}
6S|\nabla f|^2+2T^0(\nabla f,\nabla f)-\frac{36}{k_0}A(\nabla f)=k_0|\nabla f|^2.
\end{equation}
Now, obviously the assumption of the condition b) in Corollary~\ref{corol} gives us the desired estimate \eqref{shestS}, which ends the proof of Corollary~\ref{corol}.

\textbf{Acknowledgments} The research is partially supported by Contract N\textsuperscript{\underline{o}} BG051PO001-3.3.06-0052 with the Bulgarian Ministry of Education, Youth and Science and Contract N\textsuperscript{\underline{o}} 156/2013 with the University of Sofia "St. Kl. Ohridski".

\end{document}